\documentclass[letterpaper, 10 pt, conference]{ieeeconf}  

\IEEEoverridecommandlockouts                              

\usepackage{graphics} 
\usepackage{epsfig} 
\usepackage{mathptmx} 
\usepackage{amsmath} 
\usepackage{amssymb}  
\usepackage{algorithm}
\usepackage{algorithmic}

\usepackage{cite}
\usepackage[dvipsnames]{xcolor}

\usepackage[english]{babel}
\usepackage{dsfont}
\usepackage{hyperref}
\hypersetup{
    colorlinks=true,
    linkcolor=blue,
    filecolor=magenta,  
    urlcolor=cyan
    }
    
\newtheorem{theo}{Theorem}

\newtheorem{lemma}{Lemma}

\newtheorem{coro}{Corollary}

\newtheorem{remark}{Remark}

\title{\LARGE \bf
Computation of Maximal Admissible Robust Positive Invariant Sets for Linear Systems with Parametric and Additive Uncertainties
}

\author{Anchita Dey* and Shubhendu Bhasin
\thanks{
Anchita Dey is supported by the Prime Minister's Research Fellow Scheme under the Ministry of Education, Government of India.}
\thanks{Anchita Dey and Shubhendu Bhasin  are with the Department of Electrical Engineering, Indian Institute of Technology Delhi, Hauz Khas, New Delhi, Delhi 110016, India {\tt\small anchita.dey@ee.iitd.ac.in, sbhasin@ee.iitd.ac.in.}}
\thanks{*Corresponding author.}
}

\begin{document}

\maketitle
\thispagestyle{empty}
\pagestyle{empty}

\begin{abstract}
In this paper, we {{address the problem of}} computing the maximal admissible robust positive invariant (MARPI) set for discrete-time linear time-varying systems with parametric uncertainties and additive disturbances. The system state and input are subjected to hard constraints, and the system parameters and the exogenous disturbance are assumed to belong to known{{ convex polytopes}}. {{We provide necessary and sufficient conditions for the existence of the non-empty MARPI set, and explore relevant features of the set that lead to an efficient finite-time converging algorithm with a suitable stopping criterion. The analysis hinges on backward reachable sets defined using recursively computed halfspaces and the minimal RPI set.}} A numerical example is used to validate the theoretical development.
\end{abstract}


\section{Introduction}
Admissible and/or invariant sets of dynamical systems can provide valuable insights into their stability and convergence properties, as well as assist in the design of suitable controllers. A set of initial states is said to be positively invariant \cite{blanchini1999set,blanchini2008set} to a given dynamics if the subsequently evolved states belong to the same set. If this property is achieved despite the presence of uncertainties in the system model, then the set is said to be robust positive invariant (RPI). In most practical applications, the system state, input, output or their combination is subjected to constraints, thus motivating the development of maximal constraint admissible sets \cite{gilbert1991linear}, the largest possible set of initial states that evolve following the given dynamics while satisfying the imposed constraints. 

Existing literature has extensively dealt with RPI and maximal admissible sets for both linear and nonlinear systems, primarily in discrete-time setting. Some interesting applications include controller design for humanoid robots and systems in specialized Lie groups, design of control barrier functions, guaranteeing feasibility in model predictive control (MPC), etc., (see \cite{yamamoto2015maximal,weiss2014spacecraft,freire2023systematic,danielsony2023constraint, kerrigan2000invariant,pluymers2005interpolation}). One of the earliest theoretical works \cite{gilbert1991linear} details various properties of a maximal output admissible (MOA) set for an autonomous linear time-invariant (LTI) system along with an algorithm for computing the same.{{ A known non-autonomous system is considered in [3] that transforms to an autonomous system on application of state feedback control.}} The concept in \cite{gilbert1991linear} is extended to MOA sets for discrete-time LTI systems with bounded exogenous disturbance in \cite{kolmanovsky1995maximal}, and to linear time-varying (LTV) systems with polytopic uncertainty in the parameters in \cite{pluymers2005efficient}. The procedure for construction of MOA sets for nonlinear systems with exact model knowledge is provided in \cite{hirata2008exact,bravo2005computation,rachik2002maximal}, and that for bilinear systems in \cite{benfatah2021maximal}. The works in \cite{hatanaka2008probabilistic} and \cite{hatanaka2008computations} focus on MOA sets that satisfy output constraints with certain probability for time-invariant and time-varying parametric uncertainties, respectively. A few works \cite{kouramas2005minimal,blanchini1994ultimate,hirata2004maximal,kerrigan2001robust,casavola2000robust} describe constraint admissible RPI sets for linear systems with both parametric and additive uncertainties. 

In this paper, we deal with the maximal admissible RPI (MARPI) set for LTV systems with {{polytopic}} parametric uncertainties and {{polytopic}} additive disturbances. {{The LTV system is transformed to an autonomous system using state feedback control with a quadratically stabilizing feedback gain \cite{khargonekar1990robust,de1999new}. The system state and input are subjected to polytopic hard constraints.}} Unlike the sets introduced and constructed in \cite{gilbert1991linear,yamamoto2015maximal,weiss2014spacecraft,freire2023systematic,danielsony2023constraint, kerrigan2000invariant,pluymers2005interpolation,kolmanovsky1995maximal,pluymers2005efficient,hatanaka2008computations,hatanaka2008probabilistic,benfatah2021maximal,hirata2008exact,bravo2005computation,rachik2002maximal} that are positively invariant to known dynamics, {{or RPI to either parametric uncertainty or external disturbances, the proposed MARPI set is invariant to both parametric and additive uncertainties.}} This work, inspired from \cite{kolmanovsky1995maximal} and \cite{pluymers2005efficient}, seeks to rigorously analyse properties of the MARPI set that are required to design and implement a tractable algorithm. Theoretically, the computation of an MARPI set involves infinite intersections of $k$-step backward reachable sets where $k$ varies from zero to infinity \cite{kolmanovsky1995maximal,pluymers2005efficient}. The proposed algorithm is similar in spirit to the standard method of constructing MOA sets in \cite{gilbert1991linear,kolmanovsky1995maximal,pluymers2005efficient}. However, the main challenges {{that this work addresses include construction of}} the backward reachable sets, necessary and sufficient conditions for the MARPI set to be non-empty, showing that the set is obtainable with finite intersections, and proving that a stopping criterion exists for the algorithm. 

Due to the presence of both parametric and additive uncertainties, it is required to consider all possible combinations of additive disturbance for every possible value of the uncertain parameters in defining the halfspaces that form the backward reachable sets. Carrying out theoretical analysis using such set definitions becomes cumbersome. This is overcome by {{strategically defining the halfspaces in a recursive fashion}} using the vertices of the parametric uncertainty set, support functions of the disturbance set and Kronecker product. {{Different from \cite{pluymers2005efficient}, the additional challenge in this work is that of non-existence of the MARPI set due to the presence of disturbance. It is, therefore, prudent to check the existence of the MARPI set before attempting to compute it using the algorithm.}} We state necessary and sufficient conditions to check the existence of the non-empty MARPI set. {{A similar situation arises in \cite{kolmanovsky1995maximal} that considers additive disturbance only, where the minimal RPI (mRPI) set is used to obtain the existence condition. Owing to the absence of parametric uncertainty, the mRPI set in \cite{kolmanovsky1995maximal} is convex. For the case in question, where both parametric and additive uncertainties are considered, the mRPI set is a union of infinite convex sets, and the union may be non-convex. Consequently, the results in \cite{kolmanovsky1995maximal,pluymers2005efficient} do not trivially extend to the case of parametric and additive uncertainties.}} Constraint admissible RPI sets for systems with both types of uncertainties have also been addressed in \cite{kouramas2005minimal,blanchini1994ultimate,hirata2004maximal,kerrigan2001robust,casavola2000robust}. An inclusion-based mRPI set is constructed in \cite{kouramas2005minimal} whereas \cite{blanchini1994ultimate,hirata2004maximal,kerrigan2001robust,casavola2000robust} deal with MARPI sets. In \cite{blanchini1994ultimate}, the theory and computation of maximal $\lambda$-contractive set (with $0\leq \lambda <1$) is discussed and used to design a suitable feedback control. For linear systems, the designed feedback gain in \cite{blanchini1994ultimate} depends on set-induced Lyapunov function, and is therefore, not fixed for all points in the contractive set. Contrary to this, we construct the MARPI set for a given quadratically stabilizing feedback gain without considering the restrictive property of contraction. In \cite{hirata2004maximal}, a Schur stable nominal system along with small admissible perturbations as feedback that consist of stable, time-varying, memoryless and possibly nonlinear operators is considered. The parametric uncertainty is modeled using the operator assumed to be norm-bounded by one. Chapter 3 in \cite{kerrigan2001robust} lays down the theory of computing MARPI sets without providing a detailed algorithm or other associated conditions for obtaining the set. In \cite{casavola2000robust}, a robust command governor is designed based on the computation of the MARPI set. 

{{In this paper, we focus on polytopic parametric and additive uncertainties, and show in detail that the existence condition of the non-empty MARPI set can be obtained using the computable convex hull of the mRPI set \cite{kouramas2005minimal}. Provided the MARPI set is non-empty, we prove that it can be obtained with finite intersections of the suitably defined backward reachable sets. Thereafter, we show the existence of a stopping criterion leading to a tractable algorithm.}}

\textit{Notations and Definitions: }$||\cdot||_2$ and $||\cdot||_\infty$ represent $2$ and $\infty$-norms of a vector (or induced $2$ and $\infty$-norms of a matrix), respectively. {{$\emptyset$ is the empty set. The}} Minkowski sum of two sets $\mathbf{P}$ and $\mathbf{Q}$ is denoted by $\mathbf{P}\oplus\mathbf{Q}\triangleq\{p+q\;|\;p\in\mathbf{P},\;q\in\mathbf{Q}\}$. The set operations {{$\mathbf{P}\backslash \mathbf{Q}\triangleq\{p \;|\;p\in\mathbf{P},\;p\not\in\mathbf{Q}\}$}} and $V\mathbf{Q}\triangleq\{ Vq\;|\;q\in\mathbf{Q}\}$ where $V$ is a matrix. 
$0_{p}\in\mathbb{R}^{p}$ is the zero vector, $\mathds{1}_q$ is a vector in $\mathbb{R}^q$ with each element equal to $1$, and $\otimes$ denotes Kronecker product. $\mathbb{I}_p^q=\left\{p,\;p+1,\;...,\;q-1,\;q\right\}$ with integers $p$ and $q>p$. {{The}} convex hull of the set of elements $p_1,\;p_2,\;...$ is denoted by $\text{conv}\left(\{p_1,\; p_2,\;...\}\right)$. For any $p\in\mathbb{R}$, $\lfloor p\rfloor$ returns the greatest integer less than or equal to $p$. The symbol $V(i)$ denotes the $i^\text{th}$ row of a matrix/vector $V$. For two vectors $p,\;q\in\mathbb{R}^w$, $p\leq q$ or $q\geq p$ implies $p(i)\leq q(i)$ $\forall i\in\mathbb{I}_1^w$. The support function of the set $\mathbf{P}$ evaluated at the vector $q$ is represented by 
    $h_\mathbf{P}\left(q\right)\triangleq \sup_{p\in\mathbf{P}} \;\left(q^\intercal p \right)$
where `$\sup$' denotes supremum. The function $\zeta_\mathbf{P} (V)$ maps the matrix $V\in\mathbb{R}^{s\times w}$ to a vector of the support functions of $\mathbf{P}$ evaluated at the transpose of each row of $V$, i.e.,
\begin{align}\label{zetaP}\zeta_\mathbf{P} (V)=\begin{bmatrix}
    h_\mathbf{P}\left(q_1^\intercal\right) & h_\mathbf{P}\left(q_2^\intercal\right)&...&h_\mathbf{P}\left(q_s^\intercal\right)
\end{bmatrix}^\intercal\in\mathbb{R}^s\end{align}where $q_1,\; q_2,\;...,\;q_s\in\mathbb{R}^{{1\times w}}$ are the $s$ consecutive rows in $V$.
\section{Problem Formulation}\label{SecII}
Consider the following discrete-time LTV system
\begin{align}
    &x_{t+1}=A_tx_t+B_tu_t+d_t\;\;\;\;\forall t\in\mathbb{I}_0^\infty{{,}} \label{systemeqn}\\
    &x_t\in\mathbf{X}\text{ and }u_t\in\mathbf{U}{{,}}\label{systemconstraints}
\end{align}
where $x_t\in\mathbb{R}^n$ is the state, $u_t\in\mathbb{R}^m$ is the input, $A_t\in\mathbb{R}^{n\times n}$, $B_t\in\mathbb{R}^{n\times m}$ are system parameters, $d_t\in\mathbb{R}^n$ is an additive disturbance and \eqref{systemconstraints} are hard constraints on the system state and input. We define an augmented parameter $\psi_t\triangleq \begin{bmatrix}
    A_t&B_t
\end{bmatrix}\in\mathbb{R}^{n\times(n+m)}$. The parameter $\psi_t$ and the additive disturbance $d_t$ are unknown but belong to the sets $\Psi$ and $\mathbf{D}$, respectively.
For all the elements in $\Psi$ of the form $\begin{bmatrix}
    \bar{A}&\bar{B}
\end{bmatrix}$ where $\bar{A}\in\mathbb{R}^{n\times n}$ and $\bar{B}\in\mathbb{R}^{n\times m}$, there exists a common gain $K\in\mathbb{R}^{m\times n}$ such that $\left(\bar{A}+\bar{B}K\right)$ is Schur stable. In other words, any system $z_{t+1}=\left(\bar{A}+\bar{B}K\right)z_t$, where $\begin{bmatrix}
    \bar{A}&\bar{B}
\end{bmatrix}\in\Psi$, is quadratically stabilizable using the feedback gain $K$ \cite{khargonekar1990robust,de1999new}. The sets $\mathbf{X}$, $\mathbf{U}$, $\mathbf{D}$ and $\Psi$ are known convex polytopes, each containing its respective origin.

The \emph{objective} is to find the {{MARPI}} set $\mathbf{S}\subseteq \mathbf{X}$ that satisfies $A_t x_t+B_tu_t+d_t\in \mathbf{S}$ $\forall (x_t,\;\begin{bmatrix}
    A_t&B_t
\end{bmatrix},\;d_t)\in\mathbf{S}\times \Psi \times \mathbf{D}$ where $u_t=Kx_t\in\mathbf{U}$. {{The input $u_t=Kx_t$ transforms the non-autonomous system in \eqref{systemeqn} to an autonomous system with exogenous disturbance.}} Here, `admissible' refers to the satisfaction of \eqref{systemconstraints} whereas `robustness' is with respect to the uncertain parameters $A_t$, $B_t$, and the external disturbance $d_t$.
\section{Reformulation of the objective}
The polytopes $\mathbf{X},\;\mathbf{U}$ and $\mathbf{D}$ can be represented using their halfspaces, and $\Psi$ using its vertices as follows:
\begin{align}
    &\mathbf{X}\triangleq\left\{ x \;|\; F_x x \leq f_x \;,\;F_x\in\mathbb{R}^{n_x\times n},\; f_x\in\mathbb{R}^{n_x} \right\}\subseteq\mathbb{R}^n{{,}} \label{setX}\\
    &\mathbf{U}\triangleq\left\{ u \;|\; F_u u \leq f_u \;,\;F_u\in\mathbb{R}^{n_u\times m},\; f_u\in\mathbb{R}^{n_u} \right\} \subseteq\mathbb{R}^m {{,}} \label{setU}\\
    &\mathbf{D}\triangleq\left\{ d \;|\; F_d d \leq f_d \;,\;F_d\in\mathbb{R}^{n_d\times n},\; f_d\in\mathbb{R}^{n_d} \right\} \subseteq\mathbb{R}^n, \label{setD}\\
    &\Psi \triangleq \text{conv} \left( \left\{\psi^{[1]},\;\psi^{[2]},\;...,\;\psi^{[L]}  \right\}\right)\subseteq \mathbb{R}^{n\times (n+m)}, \label{setPsi}
\end{align}
where the vertex $\psi^{[i]}=\begin{bmatrix}
    A^{[i]}&B^{[i]}
\end{bmatrix}$ $\forall i\in\mathbb{I}_1^L$, $L$ is finite, and $F_x,\;f_x,\;F_u,\;f_u,\;F_d,\;f_d$ and $\psi^{[i]}$ $\forall i\in\mathbb{I}_1^L$ are known. Using the feedback gain $K$, two more polytopic sets are defined below:
\begin{align}
&{\mathbf{S}_0}\triangleq   \left\{ x\;\Big{|}\;F_0 x \leq f_0, \; F_0\triangleq\begin{bmatrix}
    F_x\\F_uK
\end{bmatrix} ,\;f_0\triangleq \begin{bmatrix}f_x\\f_u
\end{bmatrix}  \right\} \subseteq\mathbb{R}^n,\label{setZ}\\
&\Phi\triangleq \text{conv}\left(\left\{ \phi^{[1]},\;\phi^{[2]},...,\;\phi^{[L]}\right\} \right)\subseteq \mathbb{R}^{n\times n} \label{setPhi}\\
&\;\;\;\;\;\;\;\;\;\;\;\;\text{ with }\phi^{[i]}\triangleq\left(A^{[i]}+B^{[i]}K \right)\;\;\forall i\in\mathbb{I}_1^L, \nonumber
\end{align}
where each element in $\Phi$ is Schur stable since $K$ is a quadratically stabilizing feedback gain. By definition \eqref{setZ}, the elements of $\mathbf{S}_0$ satisfy \eqref{systemconstraints}, i.e., $x\in\mathbf{X}$ and $Kx\in\mathbf{U}$ $\forall x\in\mathbf{S}_0$. Therefore, our \emph{objective} can be modified to finding the MARPI set $\mathbf{S}$, where $\mathbf{S}\subseteq \mathbf{S}_0\subseteq\mathbf{X}$, {{for the dynamics  
\begin{gather}\begin{aligned}\label{invdyn}
    &x_{t+1}=\phi_t x_t+d_t\;\;\;\;\forall t\in  \mathbb{I}_0^\infty,    \\
   & x_t\in\mathbf{S}_0, \;\phi_t\triangleq \left( A_t+B_tK\right)\in\Phi\text{ and }d_t\in\mathbf{D},
    \end{aligned}\end{gather}
i.e., $x_t\in\mathbf{S}\Rightarrow x_{t+1}\in\mathbf{S}$ following \eqref{invdyn}.}} Redefining the objective in the above manner circumvents the harder task of finding the MARPI set $\mathbf{S}$ in $\mathbf{X}$ for which $K\mathbf{S}\subseteq\mathbf{U}$.

\section{Construction of the MARPI Set}

{{Consider a set ${{\mathbf{S}_1}}\triangleq\{x\;|\;\phi_{t_0} x+d\in\mathbf{S}_0,\;\;\forall (\phi_{t_0},\;d)\in\Phi\times \mathbf{D}\}=\{ x \;|\; F_0\phi_{t_0} x+F_0d\leq f_0,\;\;\forall (\phi_{t_0},\;d)\in\Phi\times \mathbf{D}\}=\{ x\;|\; {F_0\phi_{t_0}}x\leq f_0-F_0d,\;\;\forall (\phi_{t_0},\;d)\in\Phi\times \mathbf{D} \}$.}} The right-hand side of the inequality involves subtraction $f_0-F_0d \;\;\forall d\in\mathbf{D}$, from which it can be deduced that the irredundant halfspaces contributing to the set definition are the ones that correspond to the worst case disturbance. To reduce computation, instead of subtracting with all $d\in\mathbf{D}$, it is preferable to use the function $\zeta_\mathbf{D}(\cdot)$, defined in \eqref{zetaP}, which is a vector of support functions (see \textit{Notations and Definitions}) that capture the worst effect of the disturbance, and then redefine $\mathbf{S}_1=\{ x\;{|}\;F_0\phi_{t_0} x\leq f_0-\zeta_{\mathbf{D}}(F_0),\;\;\forall \phi_{t_0}\in\Phi\}$. By definition, $\mathbf{S}_1$ is the set of all states that enter $\mathbf{S}_0$ in one time step following the dynamics in \eqref{invdyn}. Similarly, we can define sets 
{{\begin{align}
&{{\mathbf{S}_k}}\triangleq \{ x\;{|}\;F_0 {{\phi_{t_{k-1}}\phi_{t_{k-2}}...\phi_{t_{0}}}} x \leq f_0-\zeta_{\mathbf{D}}(F_0)-\zeta_{\mathbf{D}}(F_0\phi_{t_{k-1}}) \nonumber\\
&\;\;\;\;\;\;\;\;\;\;\;\;\;\;\;\;\;\;\;\;-\zeta_{\mathbf{D}}(F_0\phi_{t_{k-1}}\phi_{t_{k-2}})...-\zeta_{\mathbf{D}}\left( F_0\phi_{t_{k-1}}\phi_{t_{k-2}}...\phi_{t_1} \right), \nonumber\\
& \;\;\;\;\;\;\;\;\;\;\;\;\;\;\;\;\;\;\;\;\;\;\;\;\;\;\forall \phi_{t_{k-1}},\;\phi_{t_{k-2}},...,\phi_{t_0}\in\Phi \}\;\;\;\;\;\;\forall k\in\mathbb{I}_1^\infty,\label{noco_Sk}
\end{align}}}
where $\mathbf{S}_k$ is the set of all states that enter $\mathbf{S}_0$ in $k$ time steps.


 \begin{theo}\label{theo1}
      {{The MARPI set $\mathbf{S}$ for \eqref{invdyn} is given by 
       \begin{align}
       \mathbf{S}\triangleq \bigcap_{k=0}^\infty \mathbf{S}_k,\label{Sint}\end{align}
       where each $\mathbf{S}_k$ is defined in \eqref{noco_Sk}. Also, $\mathbf{S}$ is convex.}}
  \end{theo}
\begin{proof}
{{By definition of $\mathbf{S}_k$ in \eqref{noco_Sk}, if the initial state $x_0\in\mathbf{S}_k$, then, following the dynamics in \eqref{invdyn}, the state $x_k\in\mathbf{S}_0$. Therefore, the intersection of the sets $\mathbf{S}_k$ $\forall k\in\mathbb{I}_0^\infty$ yields the set of all possible states $x_0\in\mathbf{S}_0$ such that $x_1,\;x_2,\;x_3,\;...,\;\lim_{k\rightarrow\infty}x_k \in\mathbf{S}_0$ following the dynamics in \eqref{invdyn}. Further, for any $ z\in\mathbf{S}=\mathbf{S}_0\cap\mathbf{S}_1\cap\mathbf{S}_2\cap...$, we can write
\begin{align*}
   & z\in\mathbf{S}_0\Rightarrow F_0 z\leq f_0,\\
   & z\in\mathbf{S}_1\Rightarrow F_0 (\phi_{t_0} z+d) \leq f_0 \Rightarrow \phi_{t_0} z+d \in\mathbf{S}_0,\\
   & z\in\mathbf{S}_2\Rightarrow F_0 \left(\phi_{t_1}(\phi_{t_0}z+d)+d \right)\leq f_0 \Rightarrow \phi_{t_0}z+d\in\mathbf{S}_1,\\
  &\;\;\;\;\;\;\;\;\;\;\;\;\;\;\;\;  \vdots\\
 & z\in \lim_{k\rightarrow\infty}\mathbf{S}_k\Rightarrow \phi_{t_0}z+d \in\lim_{k\rightarrow\infty}\mathbf{S}_{k-1},
\end{align*}
where $\phi_{t_0},\;\phi_{t_1}\in\Phi$ and $d\in\mathbf{D}$, i.e., $z\in\mathbf{S}\Rightarrow \phi_{t_0}z+d\in\mathbf{S}_0\cap\mathbf{S}_1\cap\mathbf{S}_2\cap...=\mathbf{S}$. Therefore, $\mathbf{S}$ is an admissible RPI set.}}

{{If $\mathbf{S}=\mathbf{S}_0$, then $\mathbf{S}$ is the maximal. If $\mathbf{S}\subset \mathbf{S}_0$, consider any element $z\in\mathbf{S}_0\backslash \mathbf{S}$. Then, $\exists$ at least one $k\in\mathbb{I}_1^\infty$ such that if $x_0=z$, the state $x_k$ will leave $\mathbf{S}_0$. This implies that any $z \in \mathbf{S}_0\backslash \mathbf{S}$ does not belong to an admissible RPI set, and therefore, $\mathbf{S}$ is the MARPI set.}}

{{Since the dynamics in \eqref{invdyn} is linear, each of the sets $\mathbf{S}_k$ can be represented as the intersection of halfspaces as defined in \eqref{noco_Sk}. Consequently, the set $\mathbf{S}$ is obtained by the intersection of infinite halfspaces. Let $z_1,\; z_2\in \mathbf{S}$. Then, $Gz_1\leq g$ and $Gz_2\leq g$ where the pair $(G,\;g)$ represents any of the halfspaces forming $\mathbf{S}$. For any $\lambda\in[0,1]$, $G\left( \lambda z_1+(1-\lambda )z_2\right)=\lambda G z_1 +(1-\lambda) G z_2\leq \lambda g+(1-\lambda)g = g$. This is true for all the halfspaces of $\mathbf{S}$, which proves that $\mathbf{S}$ is convex. }}
  \end{proof}
  
Computing the sets $\mathbf{S}_k$ as defined in \eqref{noco_Sk} is intractable since there are infinite halfspaces owing to infinite number of elements $\phi\in\Phi$. The following lemma shows that it is in fact sufficient to construct $\mathbf{S}_k$ using only the vertices of $\Phi$ instead of each individual element in $\Phi$.

\begin{lemma}\label{lemma1}
  {{ Each set $\mathbf{S}_k$ in \eqref{noco_Sk} can be redefined recursively using the vertices of $\Phi$ defined in \eqref{setPhi} as}}
   \begin{align}
   & {\mathbf{S}_k=\left\{x\;|\;F_kx\leq f_k\right\}\;\forall k\in\mathbb{I}_1^\infty,\label{Sintprevious}}\\   &{\text{where } F_k\triangleq\begin{bmatrix}
      F_{k-1}\phi^{[1]}\\F_{k-1}\phi^{[2]}\\\vdots\\F_{k-1}\phi^{[L]}
  \end{bmatrix}\in\mathbb{R}^{L^k(n_x+n_u)\times n}\label{Ak}}\\
  & {\text{and }f_k\triangleq
      \mathds{1}_{L}\otimes \left( f_{k-1}-
       \zeta_{\mathbf{D}}\left( F_{k-1}  \right) \right)\in\mathbb{R}^{L^k(n_x+n_u)}.}\label{bk}
  \end{align}
\end{lemma}
\begin{proof}
    See subsection \textit{A} in the Appendix.
\end{proof}

Although Lemma \ref{lemma1} enables recursive computation of the sets $\mathbf{S}_k$, it is still practically impossible to perform the infinite intersections in \eqref{Sint} for obtaining $\mathbf{S}$. However, \eqref{Sint}-\eqref{bk} are useful for gaining insight into some properties of $\mathbf{S}$, based on which a tractable algorithm is mentioned in Section \ref{sectrac}. 

The subtraction of the support functions in the computation of each $\mathbf{S}_k$ may result {{in $\mathbf{S}$ being empty; therefore, we propose conditions to check the existence of the non-empty MARPI set $\mathbf{S}$.}} To this end, consider
\begin{align}\label{defD}
    \mathbf{D}_k\triangleq {{\left(\bigcup_{i=1}^L \phi^{[i]}\mathbf{D}_{k-1} \right) \oplus \mathbf{D}}}\;\;\;\forall k\in\mathbb{I}_1^\infty,\;\text{and }
    \mathbf{D}_0\triangleq\{0_n\}.
\end{align}
In the limiting case when $k\rightarrow\infty$ in \eqref{defD}, we obtain $\mathbf{D}_\infty$ which is the union of infinite convex sets, and is called the mRPI set\footnote{With $x_0=\{0_n\}$ and the bounded uncertainties $\phi_t$ and $d_t$, all possible values of $\lim_{t\rightarrow\infty} x_t\in\mathbf{D}_\infty$ and further, $\phi\mathbf{D}_\infty \oplus\mathbf{D} \subseteq \mathbf{D}_\infty$ $\forall \phi\in\Phi$.}\cite{kouramas2005minimal,rakovic2005invariant}. The union may be non-convex, and computing such a set is challenging. For deriving the existence condition of a non-empty $\mathbf{S}$, we show that it suffices to use the convex hull of $\mathbf{D}_\infty$ given by $\mathbf{D}_\infty^{\text{co}}=\lim_{k\rightarrow\infty}\mathbf{D}_k^{\text{co}}$, where
{{\begin{align}\label{defDco}
    \mathbf{D}^{\text{co}}_k\triangleq \text{conv}\left(\bigcup_{i=1}^L \phi^{[i]}\mathbf{D}^{\text{co}}_{k-1} \right) \oplus \mathbf{D}\;\;\;\forall k\in\mathbb{I}_1^\infty,\;\;\mathbf{D}_0^{\text{co}}\triangleq\mathbf{D}_0.
\end{align}}}
The set $\mathbf{D}_k^{\text{co}}$ can be constructed following the algorithm developed in \cite{kouramas2005minimal} using the recursion in \eqref{defDco}. {{In the next theorem, we establish the necessary and sufficient conditions for the existence of the non-empty MARPI set.}} 

\begin{theo}\label{theo2}
{{For the system given in \eqref{invdyn},\\
\emph{(a)} the MARPI set $\mathbf{S}\neq\emptyset$ exists iff the mRPI set $\mathbf{D}_\infty\subseteq\mathbf{S}_0$.}} \\
\emph{(b)} Alternatively, let $f_{\min}\triangleq \min_{i\in\mathbb{I}_1^{n_x+n_u}} \tilde{f}_\infty(i)\in\mathbb{R}$ where $\tilde{f}_\infty\triangleq f_0-\zeta_{{{\mathbf{D}^{\text{co}}_\infty}}}( F_0 )$. The set $\mathbf{S}\neq\emptyset$ exists {{iff }}$f_{\min}\geq 0$. 
\end{theo}

\begin{proof}
{{\emph{(a)} By definition, the MARPI set $\mathbf{S}$ is required to be a subset of $\mathbf{S}_0$. If the minimal RPI set $\mathbf{D}_\infty \not\subseteq\mathbf{S}_0$, then there can not exist any other non-empty admissible RPI set in $\mathbf{S}_0$. This implies $\mathbf{S}\neq\emptyset$ exists only if $\mathbf{D}_\infty\subseteq\mathbf{S}_0$. 

Further, if $\mathbf{D}_\infty\subseteq\mathbf{S}_0$, then $\exists$ at least one admissible RPI set $\mathbb{S}\supseteq \mathbf{D}_\infty$, guaranteeing the existence of $\mathbf{S}\supseteq\mathbb{S}$. In other words, $\mathbf{D}_\infty\subseteq\mathbf{S}_0$ guarantees that each row entry in $f_k\geq 0$ $\forall k\in\mathbb{I}_1^\infty$ resulting in non-empty $\mathbf{S}_k$, and since each $\mathbf{S}_k$ contains the origin, the set $\mathbf{S}$ is non-empty and therefore, $\mathbf{S}\neq\emptyset$ exists.}}

{{\emph{(b)} Let each of the infinite convex sets whose union forms $\mathbf{D}_\infty$ be denoted by $\mathbb{D}_{[j]}$ where $j\in\mathbb{I}_0^\infty$. {\color{black}{From \emph{(a)}, $\mathbf{S}\neq\emptyset$ exists 
\newpage \vspace*{-0.2cm} \noindent iff}} ${\color{black}{\mathbf{D}_\infty\subseteq\mathbf{S}_0}}\Leftrightarrow \bigcup_{j=0}^\infty\mathbb{D}_{[j]}\subseteq\mathbf{S}_0\Leftrightarrow \text{conv}\left(\bigcup_{j=0}^\infty\mathbb{D}_{[j]} \right)\subseteq\mathbf{S}_0$ (since $\mathbf{S}_0$ is convex) $\Leftrightarrow \text{conv}( \mathbf{D}_\infty )\subseteq\mathbf{S}_0\Leftrightarrow \mathbf{D}_\infty^{\text{co}}\subseteq\mathbf{S}_0 \Leftrightarrow {\color{black}{ F_0 d\leq f_0\;\;\forall d\in}}\mathbf{D}_\infty^{\text{co}}\Leftrightarrow {\color{black}{ f_0-\zeta_{{{\mathbf{D}_\infty^{\text{co}}}}}(F_0)\geq 0_{n_x+n_u}\Leftrightarrow f_{\min}\geq 0.}}$}}
  \end{proof}
   
{{Next, we show that if $\mathbf{S}\neq\emptyset$ exists, then it can be obtained with finite intersections, unlike \eqref{Sint}.}} To achieve this, define    
  \begin{align}
\bar{F}\triangleq||F_0||_2 \;\text{, }\;{{\bar{x} }}\triangleq \max_{x\in\mathbf{S}_0}  ||x||_2  \;\text{ and }\;\phi_{\max}\triangleq \max_{i\in\mathbb{I}_1^L} ||\phi^{[i]}||_2.\label{defFbar}
\end{align}

\begin{lemma}\label{lemma2}
      If $\phi_{\max}<1$, then       
          $\frac{\ln{f_{\min}-\ln{\left(\bar{F}\bar{x}\right)}}}{\ln{\phi_{\max}}}\geq 0.$
  \end{lemma}
  \begin{proof}
    See subsection \textit{B} in the Appendix.
\end{proof}

\begin{theo}\label{theo3}
{{If the conditions for the existence of $\mathbf{S}\neq\emptyset$ given in Theorem \ref{theo2} hold, then $\exists$ a finite integer $N\geq 0$  such that $\mathbf{S}=$}} $\bar{\mathbf{S}}\triangleq\bigcap_{k=0}^N\mathbf{S}_k$, {{i.e., $\mathbf{S}$ can be finitely determined}}. 
  \end{theo}
  
\begin{proof}
{{We show that $\exists$ a finite integer $N\geq 0$ for which $\mathbf{S}_0\subseteq\mathbf{S}_k$ $\forall k\geq N$, which is sufficient to prove the theorem statement.}} Let $z\in\mathbf{S}_0$. For some integer $k\geq 0$, we can write 
\begin{align}
||F_kz||_\infty \leq  ||F_kz||_2 &\leq\max_{{{v_{k-1}\in\mathbb{I}_1^L}}} \;|| F_{k-1}\phi^{[v_{k-1}]} z ||_2 \;\left( \text{from \eqref{Ak}} \right)\nonumber\\
   \leq & \max_{{{v_{k-1},\;v_{k-2},\;...,\;v_0\;\in\mathbb{I}_1^L}}} \;|| F_{0}\phi^{[v_{k-1}]}\phi^{[v_{k-2}]}...\phi^{[v_{0}]} z ||_2 \nonumber\\
  \leq & \;\bar{F} \phi_{\max}^k \bar{x} \;\;\;\;\left(\text{from \eqref{defFbar}}\right) .\label{ubonzbar}
\end{align}
  The largest element in $F_kz$ is $||F_kz||_\infty$ which is upper bounded by $\bar{F} \phi_{\max}^k \bar{x} $. Also, note that by definition, $f_{\min}\leq \min_{i} f_k(i)$ $\forall k\in\mathbb{I}_0^\infty$, where $i\in\mathbb{I}_1^{L^k(n_x+n_u)}$. If $\bar{F} \phi_{\max}^k \bar{x}\leq f_{\min}$ for some $k$, then $z\in\mathbf{S}_k$, implying $\mathbf{S}_0\subseteq\mathbf{S}_k$. To investigate if $\exists$ a finite $k$ satisfying $\bar{F} \phi_{\max}^k \bar{x}\leq f_{\min}$, consider the two cases given below. 

\emph{(a)} {{When $\phi_{\max}<1$, $\bar{F} \phi_{\max}^k \bar{x}\leq f_{\min}\Rightarrow k\geq \frac{\ln{f_{\min}-\ln{\left(\bar{F}\bar{x}\right)}}}{\ln{\phi_{\max}}}$}} which is finite since $f_{\min}$ is finite and $\mathbf{S}_0$ is a polytope. Let $N \triangleq \left\lfloor \frac{\ln{f_{\min}-\ln{\left(\bar{F}\bar{x}\right)}}}{\ln{\phi_{\max}}} \right\rfloor$. From Lemma \ref{lemma2}, we know that $N\geq 0$. {{Therefore, $\exists$ a finite $N$ such that $\mathbf{S}_0\subseteq\mathbf{S}_k\;\forall k\geq N$,}} and $\mathbf{S}=\bigcap_{k=0}^\infty\mathbf{S}_k=\bigcap_{k=0}^N\mathbf{S}_k=\bar{\mathbf{S}}$. Note that owing to the finite number of intersections, the set $\bar{\mathbf{S}}=\mathbf{S}$ is a polytope.

\emph{(b)} {{When $\phi_{\max}\geq 1$, $\ln{\phi_{\max}}\geq 0$ due to which $\bar{F} \phi_{\max}^k \bar{x}\leq f_{\min}\not\Rightarrow k\geq \frac{\ln{f_{\min}-\ln{\left(\bar{F}\bar{x}\right)}}}{\ln{\phi_{\max}}}$.}} However, since $K$ is quadratically stabilizing, $\exists$ a Lyapunov function $V_t(x_t)=x_t^\intercal Px_t$ for the dynamics $x_{t+1}=\phi_t x_t$ $\forall \phi_t\in\Phi$. {{Let $(\cdot)'$ represent a transformed setup}} obtained using $P$, i.e., the transformed state $x'_t=P^{1/2}x_t$ and sets {{$\mathbf{S}'_0=P^{1/2}\mathbf{S}_0$, $\mathbf{D}'=P^{1/2}\mathbf{D}$ and $\Phi'=P^{1/2}\Phi P^{-1/2}$}} resulting in $\phi'_{\max}\triangleq \max_{i\in\mathbb{I}_1^L}||P^{1/2}\phi^{[i]}P^{-1/2}||_2<1$. Now, similar to the proof in \emph{(a)}, {{$\exists$ a finite $N'$ and a set ${\bar{\mathbf{S}}}'=\bigcap_{k=0}^{N'}\mathbf{S}_k'=\mathbf{S}'$, the MARPI set in the transformed setup. Since $\mathbf{S}'$ can be finitely determined, $\exists$ a finite $N$ in the original setup such that the set $\mathbf{S}=P^{-1/2}{{\mathbf{S}}}'$ can also be finitely determined.}} 
\end{proof}

Using Theorem \ref{theo3}, an implementable approach for computing $\mathbf{S}$ would be to perform the $N$ intersections by finding the sets $\mathbf{S}_k$ $\forall k\in\mathbb{I}_0^N$. The brute-force method to carry this out is to obtain $F^*\triangleq\begin{bmatrix}
        F_0^\intercal & F_1^\intercal &...&F_N^\intercal
    \end{bmatrix}^\intercal$ and $f^*\triangleq \begin{bmatrix}
        f_0^\intercal &f_1^\intercal &... &f_N^\intercal
    \end{bmatrix}^\intercal,$ with the intersection given by $\mathbf{S}=\left\{ x\;|\;F^*x\leq f^*  \right\}$. However, even with a finite $N$, the number of rows in $F_k$ and $f_k$ grow exponentially with $k$ (see \eqref{Ak}, \eqref{bk}), resulting in high computational burden. 

{{Also, note that a finite $N$ is used in Theorem \ref{theo3} for the sake of proving that $\mathbf{S}$ is obtainable with finite intersections.}} It may be possible to obtain the set $\mathbf{S}$ with less than $N$ number of intersections since the proof of Theorem \ref{theo3} is performed using a conservative upper bound of the elements of $F_kz$ and a conservative lower bound of the elements of $f_k$. Instead of using $F^*$ and $f^*$, motivated by \cite{pluymers2005efficient}, we propose a more elegant and less computationally burdensome algorithm for computing the set $\mathbf{S}$. 

\section{Tractable method to compute the MARPI Set}\label{sectrac}
In this section, we present a tractable algorithm for computing the MARPI set {{without knowing the value of $N$.}} We begin by developing a stopping criterion for performing the finite set intersections.

\begin{theo}\label{theo4}
{{Suppose the conditions for the existence of $\mathbf{S}\neq\emptyset$ given in Theorem \ref{theo2} are satisfied. Let}} ${\tilde{\mathbf{S}}}\triangleq\bigcap_{{{k}}{\color{black}{=0}}}^{{r}} \mathbf{S}_k$, {{where $r\in\mathbb{I}_0^N$}}. {{The set $\tilde{\mathbf{S}}\subseteq \mathbf{S}_{{{r}}+1}$ iff $\tilde{\mathbf{S}}=\mathbf{S}$}}.
\end{theo}
\begin{proof}
Using the definitions of $\mathbf{S}_k$ in \eqref{noco_Sk} and $\tilde{\mathbf{S}}$, we know that {{$\tilde{\mathbf{S}}\subseteq \mathbf{S}_1$ implies all elements of $\tilde{\mathbf{S}}$ enter $\mathbf{S}_0$ after one time step evolution following the dynamics in \eqref{invdyn},}} i.e., ${{\phi_{t_0}\tilde{\mathbf{S}} \oplus\mathbf{D} }}\subseteq\mathbf{S}_0$ $\forall \phi_{t_0}\in\Phi$. Also, $\tilde{\mathbf{S}}\subseteq \mathbf{S}_2$ $\Rightarrow \phi_{t_1}\phi_{t_0}\tilde{\mathbf{S}}\oplus \phi_{t_1}\mathbf{D}\oplus \mathbf{D}\subseteq \mathbf{S}_0 \Rightarrow \phi_{t_1} \left( \phi_{t_0}\tilde{\mathbf{S}}\oplus \mathbf{D} \right)\oplus\mathbf{D}\subseteq \mathbf{S}_0\Rightarrow \phi_{t_0}\tilde{\mathbf{S}}\oplus \mathbf{D}\subseteq\mathbf{S}_1 $ (by definition of $\mathbf{S}_1$). Similarly, $\tilde{\mathbf{S}}\subseteq \mathbf{S}_r\Rightarrow  \phi_{t_{r-1}}\phi_{t_{r-2}}...\phi_{t_{1}} \left( \phi_{t_0}\tilde{\mathbf{S}}\oplus \mathbf{D} \right) \oplus \phi_{t_{r-1}}...\phi_{t_{1}}\mathbf{D} \oplus ... \oplus\phi_{t_{r-1}} \mathbf{D}\oplus \mathbf{D}\subseteq\mathbf{S}_0 \Rightarrow \phi_{t_0}\tilde{\mathbf{S}}\oplus \mathbf{D}\subseteq \mathbf{S}_{r-1}$ $\forall \phi_{t_{r-1}},\;\phi_{t_{r-2}},...,\;\phi_{t_0}\in\Phi$. {{To summarize, $\tilde{\mathbf{S}}=\bigcap_{k=0}^r \mathbf{S}_k \Rightarrow\phi \tilde{\mathbf{S}}\oplus\mathbf{D}\subseteq \bigcap_{k=0}^{r-1}\mathbf{S}_k$ $\forall \phi\in\Phi$.}}

Now, following the same analysis, if $\tilde{\mathbf{S}}\subseteq \mathbf{S}_{r+1}$, then $\phi \tilde{\mathbf{S}}\oplus \mathbf{D}\subseteq \mathbf{S}_r$ $\forall \phi\in\Phi$ implying $\phi\tilde{\mathbf{S}}\oplus \mathbf{D}\subseteq \bigcap_{k=0}^r \mathbf{S}_k=\tilde{\mathbf{S}}$ $\forall \phi\in\Phi$. {{Using $\phi\tilde{\mathbf{S}}\oplus\mathbf{D}\subseteq \tilde{\mathbf{S}}$ $\forall\phi\in \Phi$,}} we can write $\forall k\in\mathbb{I}_{r+2}^\infty$, 
\begin{align}
&\;\;\;\;\; \phi_{t_{k-1}}...\phi_{t_0}\tilde{\mathbf{S}}\oplus  \phi_{t_{k-1}} ... \phi_{t_1}  \mathbf{D}\oplus...\oplus \phi_{t_{k-1}}\mathbf{D}\oplus\mathbf{D}\label{maggi2}\\
&=\phi_{t_{k-1}}...\phi_{t_1}\left(\phi_{t_0}\tilde{\mathbf{S}}\oplus\mathbf{D}\right) \oplus \phi_{t_{k-1}} ... \phi_{t_2}  \mathbf{D} \oplus ...\oplus \phi_{t_{k-1}}\mathbf{D}\oplus\mathbf{D}\nonumber\\
&\subseteq \phi_{t_{k-1}} ... \phi_{t_1}\tilde{\mathbf{S}}\oplus  \phi_{t_{k-1}} ... \phi_{t_2}  \mathbf{D}\oplus...\oplus \phi_{t_{k-1}}\mathbf{D}\oplus\mathbf{D} \nonumber\\
&=\phi_{t_{k-1}} ... \phi_{t_2}\left(\phi_{t_1}\tilde{\mathbf{S}}\oplus\mathbf{D} \right)\oplus  \phi_{t_{k-1}} ... \phi_{t_3}  \mathbf{D} \oplus...\oplus \phi_{t_{k-1}}\mathbf{D}\oplus\mathbf{D} \nonumber\\
&\subseteq...=\phi_{t_{k-1}}\left(\phi_{t_{k-2}}\tilde{\mathbf{S}}\oplus\mathbf{D} \right)\oplus \mathbf{D}\subseteq \phi_{t_{k-1}}\tilde{\mathbf{S}}\oplus \mathbf{D} \subseteq\tilde{\mathbf{S}} \subseteq \mathbf{S}_0 ,\label{maggi1}
\end{align}
implying $\tilde{\mathbf{S}}\subseteq \mathbf{S}_k\;\;\forall k\in\mathbb{I}_{r+2}^\infty$ since in $k$ steps the elements of $\tilde{\mathbf{S}}$ enter $\mathbf{S}_0$ (see \eqref{maggi2} and \eqref{maggi1}). Therefore, $\tilde{\mathbf{S}}\subseteq \mathbf{S}_{r+1}$ implies $\tilde{\mathbf{S}}=\bigcap_{k=0}^r \mathbf{S}_k=\bigcap_{k=0}^\infty \mathbf{S}_k=\mathbf{S}$ (since $\tilde{\mathbf{S}}\subseteq \mathbf{S}_k\;\;\forall k\in\mathbb{I}_{r+1}^\infty$).

Also, $\tilde{\mathbf{S}}=\mathbf{S}\Rightarrow \tilde{\mathbf{S}}=\tilde{\mathbf{S}}\cap \mathbf{S}_{r+1}\cap\mathbf{S}_{r+2}\cap...\Rightarrow \tilde{\mathbf{S}}\subseteq \mathbf{S}_k\;\forall k\in\mathbb{I}_{r+1}^\infty$, which completes the proof.
\end{proof}
  
To compute $\mathbf{S}$, initialize a set $\hat{\mathbf{S}}=\mathbf{S}_0$. Define the precursor set of $\hat{\mathbf{S}}$ as $\text{Pre}\left( \hat{\mathbf{S}} \right)\triangleq \{ x\;|\;\phi^{[i]}x+d\in \hat{\mathbf{S}},\;\;\forall(i,d)\in\mathbb{I}_1^L\times\mathbf{D} \}$. Clearly,
$\text{Pre}\left(\hat{\mathbf{S}} \right)=\text{Pre}\left({\mathbf{S}_0}\right)=\mathbf{S}_1$. At iteration 1, update $\hat{\mathbf{S}}\leftarrow\hat{\mathbf{S}}\cap \text{Pre}\left(\hat{\mathbf{S}} \right) $.  The updated $\hat{\mathbf{S}}$ is the maximal set of all states in $\mathbf{S}_0$ that stay in $\mathbf{S}_0$ after one time step evolution following \eqref{invdyn}. Now, $\text{Pre}\left(\hat{\mathbf{S}}\right)\subseteq \mathbf{S}_2$. Updating $\hat{\mathbf{S}}$ in the $2^\text{nd}$ iteration as $\hat{\mathbf{S}}\leftarrow \hat{\mathbf{S}}\cap \text{Pre}\left(\hat{\mathbf{S}}\right)$ gives the maximal set of all states in $\mathbf{S}_0$ that stay in $\mathbf{S}_0$ after one as well as two time steps evolution. For the newly obtained set $\hat{\mathbf{S}}$, we have, $\text{Pre}\left(\hat{\mathbf{S}}\right)\subseteq\mathbf{S}_3$. 

In general, the $\text{Pre}\left(\hat{\mathbf{S}}\right)\subseteq\mathbf{S}_k$ where $\hat{\mathbf{S}}$ is obtained at iteration $k-1$. The recursive process of reassigning $\hat{\mathbf{S}}\leftarrow \hat{\mathbf{S}}\cap \text{Pre}\left(\hat{\mathbf{S}}\right)$ is continued until $\hat{\mathbf{S}}\subseteq \text{Pre}\left(\hat{\mathbf{S}}\right)$. The set $\hat{\mathbf{S}}$ obtained when the condition $\hat{\mathbf{S}}\subseteq \text{Pre}\left(\hat{\mathbf{S}}\right)$ is satisfied yields the MARPI set $\mathbf{S}$ as proved in the following corollary.
\begin{coro}\label{coro2}
If $\hat{\mathbf{S}}\subseteq \text{Pre}\left(\hat{\mathbf{S}}\right)$, then $\hat{\mathbf{S}}=\mathbf{S}$.
\end{coro}
\begin{proof}
   Denote the set $\hat{\mathbf{S}}$ obtained at iteration $\bar{N}$ by $\hat{\mathbf{S}}_{\left(\bar{N}\right)}$. By construction, $\hat{\mathbf{S}}_{\left(\bar{N}\right)}=\bigcap_{k=0}^{\bar{N}}\mathbf{S}_k$. If $\hat{\mathbf{S}}_{\left(\bar{N}\right)}\subseteq \text{Pre}\left(\hat{\mathbf{S}}_{\left(\bar{N}\right)}\right)$, then \newpage \vspace*{-0.2cm} \noindent$\hat{\mathbf{S}}_{\left(\bar{N}\right)}\subseteq \mathbf{S}_{\bar{N}+1}$ $\left(\text{since } \text{Pre}\left(\hat{\mathbf{S}}_{\left(\bar{N}\right)}\right)\subseteq \mathbf{S}_{\bar{N}+1}\right)$. Therefore, $\hat{\mathbf{S}}_{\left(\bar{N}\right)}=\bigcap_{k=0}^{\bar{N}}\mathbf{S}_k$ and $\hat{\mathbf{S}}_{\left(\bar{N}\right)}\subseteq \mathbf{S}_{\bar{N}+1}$. This is similar to the statement of Theorem \ref{theo4}, thus, yielding $\hat{\mathbf{S}}_{\left(\bar{N}\right)}=\mathbf{S}$, the MARPI set.  
\end{proof}

\begin{algorithm}[b!]
\caption{Computing the MARPI set $\mathbf{S}$}\label{algo1}
\begin{algorithmic}[1]
\REQUIRE $F_0$ and $f_0$ (defined in \eqref{setZ}), $\phi^{[j]}\;\forall j\in\mathbb{I}_1^L$ (defined in \eqref{setPhi}) and $\mathbf{D}$ (defined in \eqref{setD}).\;
\ENSURE $\mathbf{S}$ 
\STATE $\hat{F}\triangleq F_0$ and $\hat{f}\triangleq f_0$.
\REPEAT
\STATE $\hat{F}_1\triangleq\begin{bmatrix}
    \hat{F}\phi^{[1]}\\\hat{F}\phi^{[2]}\\\vdots\\\hat{F}\phi^{[L]}
\end{bmatrix}$, $\hat{f}_1\triangleq \mathds{1}_L \otimes \left( 
            \hat{f}-\zeta_\mathbf{D}(\hat{F})\right)
         $, $\rho_0\triangleq $ number of rows in $\hat{F}$ and $\rho_1\triangleq$ number of rows in $\hat{F}_1$.
    \FOR{$i=1$ \TO $\rho_1$} 
       \STATE Solve the following LP problem to find $\gamma$. 
       \begin{align}\gamma\triangleq&\max_{z\in \mathbb{R}^n}\;\;(\alpha z- \beta) \;\;\;\;
       \text{subject to }\;\hat{F} z\leq\hat{f},
       \label{LPopt}\end{align}where $\alpha\triangleq\hat{F}_1(i)$ and $\beta\triangleq\hat{f}_1(i)$.
       \IF{$\gamma >0$}    \STATE $\hat{F}\gets \begin{bmatrix}
           \hat{F}\\ \alpha
       \end{bmatrix}$ and $\hat{f}\gets \begin{bmatrix}
           \hat{f}\\ \beta
       \end{bmatrix}$
       \ENDIF
   \ENDFOR
   \UNTIL{number of rows in $\hat{F}=\rho_0$.}
   \STATE $\mathbf{S}\triangleq\{x\;|\;\hat{F}x\leq \hat{f}\;\}$
\end{algorithmic}
\end{algorithm}
The steps to implement this process of updating $\hat{\mathbf{S}}$ are given in Algorithm \ref{algo1}. This technique also involves repeated intersections. However, it uses the $\text{Pre}(\cdot)$ operation that computes only the one-step backward reachable set at each iterative step, thus, reducing the amount of computation. In fact, the computational burden can be further reduced by running a simple linear programming (LP) problem (see \eqref{LPopt} in Algorithm \ref{algo1}) that takes into account only the irredundant halfspaces of the intersection at each iteration.

\begin{remark}
{{The value of $r$ or $N$ is not required in Algorithm \ref{algo1}.}} Step 3 computes the halfspaces of the precursor set whereas the intersection is performed in Steps 4 to 9. Due to the maximization in \eqref{LPopt}, a $\gamma\leq 0\Rightarrow \alpha z \leq \beta$ $\forall z\in\{ x\;|\;\hat{F}x\leq\hat{f}\;\}$, and hence, that $(\alpha,\;\beta)$ pair is not considered since it forms a redundant halfspace in the intersection.
\end{remark}

\begin{remark}{{Instead of (or in addition to) the state and input constraints in \eqref{systemconstraints}, it is possible to have output constraints.}} Let the output be $y_t=Cx_t+Du_t$, where $y_t\in\mathbb{R}^o$, $C\in\mathbb{R}^{o\times n}$ and $D\in\mathbb{R}^{o\times m}$ with known $C$ and $D$. If the invariant set has to satisfy only the output constraint $y_t\in\mathbf{Y}\triangleq\left\{  y\;|\;F_yy\leq f_y,\;F_y\in\mathbb{R}^{n_y\times o},\; f_y\in\mathbb{R}^{n_y} \right\}$, then Algorithm \ref{algo1} can be used with $F_0\triangleq F_y(C+DK)$ and $f_0\triangleq f_y$, and the resulting set $\mathbf{S}$ is the MOARPI set. In presence of both output and input constraints, $F_0\triangleq \begin{bmatrix}
    F_y(C+DK)\\F_uK
\end{bmatrix}$ and $f_0\triangleq \begin{bmatrix}
    f_y\\f_u
\end{bmatrix}$.
\end{remark}
\begin{remark}
In MPC, the terminal set is typically used to establish recursive feasibility and stability \cite{kerrigan2000invariant,lorenzen2019robust,dey2022adaptive}. To reduce the prediction horizon length and consequently, the computational burden in robust MPC, it is desired that the set is MARPI. An algorithm for generating such a terminal set is mentioned in \cite[Appendix~A.1]{lorenzen2019robust} for systems with both additive and parametric uncertainties. However, the algorithm depends on the choice of a basic polytope ($\mathbb{X}_0$ in \cite{lorenzen2019robust}) for robust tubes \cite{langson2004robust} and may result in a subset of the actual MARPI set. A better approach may be to first find the MARPI set, and then use it as the terminal set as well as the basic polytope for generating the tubes. 
\end{remark}
\begin{remark}
{{The proposed theory is trivially extendable for a time-varying $K=K_t\in\mathbf{K}$ that is common for all elements in $\Psi$ at time $t$ (or a parameter-dependent $K=K(\psi_t)\in\mathbf{K}$), where $\mathbf{K}$ is a known polytope, provided the resulting set $\Phi\ni A_t+B_tK,$ $\forall \begin{bmatrix}
        A_t & B_t
\end{bmatrix}\in\Psi$ and $\forall K\in\mathbf{K}$, is a convex polytope with each element being Schur stable and $\phi_{\max}<1$. }}
\end{remark}

\section{Simulation Result}
 \begin{figure}[t!]
      \vspace{0.23cm}\centering
     \framebox{\parbox{3in}{\includegraphics[scale=0.418]{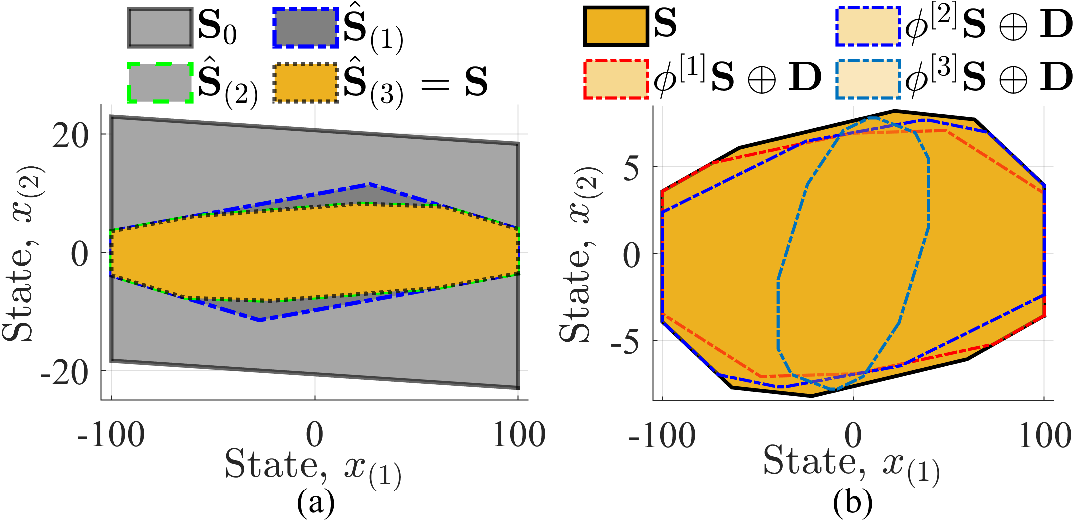}}}
      \caption{Here $x=[x_{(1)}\;\;x_{(2)} ]^\intercal$. (a) The sets $\hat{\mathbf{S}}_{(i)}$ converge to the MARPI set $\mathbf{S}$ in $2$ iterations. No new halfspaces are added in the $3^\text{rd}$ iteration. (b) The invariance property of $\mathbf{S}$ is verified since $\phi^{[j]}\mathbf{S}\oplus \mathbf{D}\subset \mathbf{S}$ $\forall j\in\mathbb{I}_1^3$.}       \label{fig1}
   \end{figure}
Consider the $2^\text{nd}$ order LTV system: $x_{t+1}=A_tx_t+B_tu_t+d_t$ with $x_t,\;d_t\in\mathbb{R}^2$, $u_t\in\mathbb{R}$, $||x_t||_\infty \leq 100$, $||u_t||_\infty \leq 100$, $||d_t||_\infty\leq 2 $ and $\begin{bmatrix}
    A_t&B_t
\end{bmatrix}\in \Psi=\text{conv}\left( \left\{ \psi^{[1]},\;\psi^{[2]},\;\psi^{[3]}\right\} \right)$\footnote{$\psi^{[1]}=\begin{bmatrix}
    0.6 & 2 & -1\\-0.1 & -6.3 & -1.2
\end{bmatrix}$, $\psi^{[2]}=\begin{bmatrix}
    0.9 & 2.11 & 2.5\\-0.13 & -5.2 & -0.96
\end{bmatrix}$ and $\psi^{[3]}=\begin{bmatrix}
    -0.3 & 2.05 & 0\\-0.12 & -4.12 & -1
\end{bmatrix}.$} $\forall t\in\mathbb{I}_0^\infty$. 

Figure \ref{fig1} is obtained with a quadratically stabilizing feedback gain $K=\begin{bmatrix}
    -0.1112 & -4.8498
\end{bmatrix}$. The sets $\hat{\mathbf{S}}$ at each iterative step converging to the set $\mathbf{S}$ are shown in Fig. \ref{fig1}(a). The algorithm converges in $3$ iterations and the set $\mathbf{S}$ can be represented using only $10$ irredundant halfspaces. In Fig. \ref{fig1}(b), we see that $\phi^{[j]}\mathbf{S}\oplus\mathbf{D}\subset\mathbf{S}$ $\forall j\in\mathbb{I}_1^3$ which verifies the robust positive invariance property of $\mathbf{S}$.  
\section{Conclusion}
The paper extends the results of \cite{kolmanovsky1995maximal,pluymers2005efficient} to MARPI set for state and input constrained discrete-time LTV systems having both {{polytopic}} parametric and additive uncertainties. A {{verifiable}} necessary and sufficient condition for the existence of the non-empty MARPI set is {{obtained using the convex hull of the mRPI set. The MARPI set, if exists, is shown to be computable in finite time using a suitable stopping criterion}} leading to a tractable algorithm with possible applications in robust MPC, safe controller design using control barrier functions, etc. Future work will consider finding MARPI sets for systems with nonconvex polyhedral constraints \cite{perez2011maximal}, and also enlarging the domain of attraction for MPC by computing precursor sets which need not be invariant \cite{limon2005enlarging}.

\section*{Appendix}


\subsection{Proof of Lemma \ref{lemma1}}\label{appen1}
{{Let $\mathbb{S}_1\triangleq \{ x\;|\; F_1 x\leq f_1  \}=\{ x\;|\; F_0\phi^{[i]}x\leq f_0 -\zeta_{\mathbf{D}}(F_0),\;\;\forall i\in\mathbb{I}_1^L\} $. Since $\mathbf{S}_1$ is defined $\forall\phi\in\Phi$, we have $\mathbb{S}_1\subseteq \mathbf{S}_1$. Next, assume $\mathbf{S}_1\not\subseteq \mathbb{S}_1$. Then, $\exists$ a $z\in\mathbb{S}_1\backslash \mathbf{S}_1$ that satisfies $F_0\phi^{[i]}z\leq f_0- \zeta_{\mathbf{D}}(F_0)$ $\forall i\in\mathbb{I}_1^L$. However, each $\phi\in\Phi$ can be represented as the convex hull of the vertices of $\Phi$, i.e., $\exists\; \lambda^{[i]}\in[0,1]$ for $i=1,2,...,L$ such that $\phi=\sum_{i=1}^L \lambda^{[i]}\phi^{[i]}$ and $\sum_{i=1}^L\lambda^{[i]}=1$. This allows us to write $ \sum_{i=1}^L \lambda^{[i]} \left(F_0  \phi^{[i]} z \right)\leq \sum_{i=1}^L \lambda^{[i]} \left( f_0-  \zeta_{\mathbf{D}}(F_0) \right)\Rightarrow F_0 \phi z \leq f_0 -\zeta_\mathbf{D} (F_0)\Rightarrow z\in\mathbf{S}_1$. This contradicts our assumption. Therefore, $\mathbf{S}_1\subseteq\mathbb{S}_1$ implying $ \mathbb{S}_1=\mathbf{S}_1$. Continuing similarly, it can be shown that $\mathbb{S}_k\triangleq \{ x\;|\;F_kx\leq f_k\}=\{ x\;|\; F_{k-1}\phi^{[i]}x\leq f_{k-1}-\zeta_{\mathbf{D}}(F_{k-1}),\;\forall i\in\mathbb{I}_1^L\}=\mathbf{S}_k$ $\forall k\in\mathbb{I}_2^\infty$, which concludes the proof.}}

{{ The definition of $F_k$ in \eqref{Ak} is easy to follow. For comprehending $f_k$ in \eqref{bk}, we explain the construction of $f_1$ and $f_2$. In $\mathbf{S}_1$, we have $f_0-\zeta_{\mathbf{D}}(F_0)$ that needs be repeated $L$ times for $L$ vertices, and so, $f_1=\mathds{1}_{L}\otimes \left(f_0 - \zeta_{\mathbf{D}}\left( F_0 \right)\right)$. For $\mathbf{S}_2$, we need $f_0-\zeta_\mathbf{D}(F_0)-\zeta_{\mathbf{D}}(F_0\phi_{t_1})$ corresponding to each $F_0\phi_{t_1}\phi_{t_0}$ $\forall \phi_{t_1},\;\phi_{t_0}\in\Phi$. Considering the vertex-based definition, we require $L^2$ times repetition of $f_0-\zeta_{\mathbf{D}}(F_0)$, and $L$ times repetition of $\zeta_{\mathbf{D}}\left(F_0\phi^{[i]}\right)$ $\forall i\in\mathbb{I}_1^L$ to obtain $f_2=\mathds{1}_{L^2}\otimes \left(f_0 - \zeta_{\mathbf{D}}\left( F_0  \right)\right)-\left(\mathds{1}_{L} \otimes \zeta_{\mathbf{D}}\left( F_1  \right)\right)=\mathds{1}_L \otimes \left( f_1-\zeta_{\mathbf{D}}(F_1)  \right)$. Following similar steps yields definition \eqref{bk} of $f_k$ $\forall k\in\mathbb{I}_1^\infty$.}}
\subsection{Proof of Lemma \ref{lemma2}}\label{appen2}
Let $\bar{\mathbf{S}}_0\triangleq  {\{} z\in\mathbb{R}^n\; {|} \;||z||_2\leq \min_{i\in\mathbb{I}_1^{n_x+n_u}} f_0(i)/{\bar{F}}  {\}} \text{ and }\bar{\mathbf{S}}_\infty\triangleq {\{} z\in\mathbb{R}^n\; {|}\;||z||_2\leq{f_{\min}}/{\bar{F}} {\}}.$
Since $f_{\min}\leq \min_{i\in\mathbb{I}_1^{n_x+n_u}} f_0(i)$, using \eqref{setZ} and \eqref{defFbar}, we can write $\mathbf{S}_0\supseteq  \bar{\mathbf{S}}_0\supseteq \bar{\mathbf{S}}_\infty$      which guarantees that $\exists\;x\in\mathbf{S}_0$ that satisfies  $||x||_2\geq {f_{\min}}/{\bar{F}}$$\Rightarrow \bar{x}\geq {f_{\min}}/{\bar{F}}\Rightarrow \ln{{f_{\min}}-\ln{\left(\bar{F}\bar{x}\right)}}\leq 0\Rightarrow  \left( \ln{f_{\min}}-\ln{\left(\bar{F}\bar{x}\right)}\right)/\ln{\phi_{\max}}\geq 0$ ($\because \ln{\phi_{\max}}<0$ for $\phi_{\max}<1$).
 
\bibliographystyle{IEEEtran}
\bibliography{IEEEabrv,reference}

\end{document}